\documentclass{article}
\usepackage{amssymb}
\usepackage{amsfonts}
\usepackage{amsmath}

\newtheorem{theorem}{Theorem}[section]
\newtheorem{corollary}[theorem]{Corollary}
\newtheorem{lemma}[theorem]{Lemma}
\newtheorem{proposition}[theorem]{Proposition}

\newtheorem{definition}[theorem]{Definition}
\newtheorem{example}[theorem]{Example}

\newtheorem{remark}[theorem]{Remark}

\newenvironment{proof}[1][Proof]{\noindent\textbf{#1.} }{\ \rule{0.5em}{0.5em}}

\input{tcilatex}
\begin{document}

\title{Rigidity of matrix group actions on $\mathrm{CAT(0)}$ spaces with
possible parabolic isometries and uniquely arcwise connected spaces}
\author{Shengkui Ye}
\maketitle

\begin{abstract}
It is well-known that $\mathrm{SL}_{n}(\mathbf{Q}_{p})$ acts without fixed
points on an $(n-1)$-dimensional $\mathrm{CAT}(0)$ space (the affine
building). We prove that $n-1$ is the smallest dimension of $\mathrm{CAT}(0)$
spaces on which matrix groups act without fixed points. Explicitly, let $R$
be an associative ring with identity and $E_{n}^{\prime }(R)$ the extended
elementary subgroup. Any isometric action of $E_{n}^{\prime }(R)$ on a
complete $\mathrm{CAT(0)}$ space $X^{d}$ of dimension $d<n-1$ has a fixed
point. Similar results are discussed for automorphism groups of free groups.
Furthermore, we prove that any action of $\mathrm{Aut}(F_{n}),n\geq 3,$ on a
uniquely arcwise connected space by homeomorphisms has a fixed point.
\end{abstract}

\section{Introduction\label{section1}}

Let $p$ be a prime, $\mathbf{Q}_{p}$ the $p$-adic number field and $\mathrm{%
SL}_{n}(\mathbf{Q}_{p})$ the special linear group. It is well-known that $%
\mathrm{SL}_{n}(\mathbf{Q}_{p})$ acts without fixed points on the affine
building $X,$ which is an $(n-1)$-dimensional $\mathrm{CAT}(0)$ space. In
the article, we prove that $n-1$ is the smallest dimension of $\mathrm{CAT}%
(0)$ spaces on which $\mathrm{SL}_{n}(\mathbf{Q}_{p})$ acts without fixed
points. More generally, let $R$ be an associative ring with identity and $%
E_{n}^{\prime }(R)<\mathrm{GL}_{n}(R)$ the extended elementary subgroup (cf.
Section \ref{elem}), which is an analog of $\mathrm{SL}_{n}(\mathbf{Q}_{p})$
for general rings. Out first result is the following.

\begin{theorem}
\label{1.1}Any isometric action of the extended elementary group $%
E_{n}^{\prime }(R)$ on a complete $\mathrm{CAT}(0)$ space $X$ of dimension $%
d<n-1$ has a fixed point.
\end{theorem}

It is a classical result of Serre \cite{Se} that any isometric action of $%
\mathrm{SL}_{n}(\mathbb{Z})$ $(n\geq 3)$ on a simplicial tree has a fixed
point. Farb \cite{fa} considered a high-dimensional analog. For a reduced,
irreducible root system $\Phi $ has rank $r\geq 2$ and a finitely generated
commutative ring $R,$ let $E(\Phi ,R)$ be the elementary subgroup of the
Chevalley group $G(\Phi ,R).$ Farb \cite{fa} proves that any action of $%
E(\Phi ,R)$ on a complete $\mathrm{CAT}(0)$ space $X$ of dimension $d<r-1$
by \emph{semisimple} isometries has a fixed point. This gives also a
generalization of a result obtained by Fukunaga \cite{fu} concerning groups
acting on trees$.$ Let $\Sigma _{g}$ be the orientable closed surface and $%
\mathrm{MCG}(\Sigma _{g})$ the mapping class group. Bridson \cite{bdm, brid}
proves any action of $\mathrm{MCG}(\Sigma _{g})$ on a complete CAT(0) space $%
X$ of dimension $d<g-1$ by \emph{semisimple} isometries has a fixed point.
The author \cite{ye} obtains similar results for actions of the elementary
subgroup $E_{n}(R)$ and the quadratic elementary subgroup $EU_{2n}(R,\Lambda
)$ for general rings. However, most of results are proved for semisimple
actions. The group action on \textrm{CAT(0) }spaces of\textrm{\ }%
automorphism groups of free groups is studied by Bridson \cite{brid2,brid11}
and Varghese \cite{var}. Barnhill \cite{bar} considers the property $\mathrm{%
FA}_{n}$ for Coxeter groups.

In this article, we study group actions on \textrm{CAT(0) }spaces with
possible parabolic isometries. Note that the semisimple isometries could be
very different from generic isometries. For example, any action of the
special linear group $\mathrm{SL}_{n}(\mathbb{Z})$ $(n\geq 3)$ on a complete 
\textrm{CAT(0) }space $X$ by semisimple isometries has a fixed point, while $%
\mathrm{SL}_{n}(\mathbb{Z})$ acts on the symmetric space $\mathrm{SL}_{n}(%
\mathbb{R})/SO(n)$ properly and thus without fixed points. It is an open
problem in geometric group theory to study the minimal dimensions of $%
\mathrm{CAT(0)}$ spaces on which matrix groups or automorphism groups of
free groups act without fixed points (see Bridson \cite{bridsonprob},
Question 9.2).

In order to state our results easily, it is better to define the following.

\begin{definition}
(relative property $\mathcal{FA}_{n}$) Let $H$ be a subgroup of a group $G.$
The pair $(G,H)$ has the relative property $\mathcal{FA}_{n}$ (resp. $s%
\mathcal{FA}_{n}$) if whenever $G$ acts on an $n$-dimensional complete $%
\mathrm{CAT(0)}$ space $X$ by isometries (resp. semisimple isometries), the
subgroup $H$ has a fixed point.
\end{definition}

We call a group $G$ has property $\mathcal{FA}_{n}$ if $(G,G)$ has the
relative property $\mathcal{FA}_{n}.$ It should be pointed that we don't
require the isometries to be semisimple in the definition of $\mathcal{FA}%
_{n}.$ Note that the property $s\mathcal{FA}_{n}$ is the same as the strong
property $\mathrm{FA}_{n}$ considered by Farb \cite{fa}.

\begin{theorem}
\label{1.2}Let $R$ be an associative ring and $E_{n}^{\prime }(R)$ the
extended elementary subgroup. Then $(R^{n}\rtimes E_{n}^{\prime }(R),R^{n})$
has the relative property $\mathcal{FA}_{n-1}.$
\end{theorem}

Since $E_{n}(\mathbb{Z}[\frac{1}{p}])\ltimes \mathbb{Z}[\frac{1}{p}]^{n}$ as
a subgroup of $\mathrm{GL}_{n+1}(\mathbf{Q}_{p})$ acts on an affine building 
$X$ of dimension $n$ without fixed points, the dimension in Theorem \ref{1.2}
is sharp.

Let $\mathrm{Aut}(F_{n})$ be the automorphism group of the rank-$n$ free
group $F_{n}.$

\begin{theorem}
\label{1.3}$(\mathrm{Aut}(F_{n}),\mathrm{Aut}(F_{2}))$ has the relative
property $\mathcal{FA}_{n-3}$ for any $n\geq 4.$
\end{theorem}

\begin{corollary}
\label{1.4}Any isometric action of $\mathrm{Aut}(F_{n})$ $(n\geq 4)$ on a
complete $\mathrm{CAT(0)}$ space $X$ of dimension $d\leq 2[\frac{n}{3}]$ has
a fixed point, i.e. $\mathrm{Aut}(F_{n})$ has property $\mathcal{FA}%
_{2[n/3]}.$ Here $[n/3]$ is the integer part.
\end{corollary}

The previous result was also observed by Bridson \cite{brid2}. The case of $%
d\leq 2[n/2]$ is already known by Varghese \cite{var}.

The proofs of the above theorems are based on Helly's theorem and the theory
of Coxeter groups. The method has further applications to group actions on
1-dimensional spaces, like dendrites and unique arcwise connected spaces.
Recall that a dendrite $X$ is a connected compact metrizable space
(continuum) such that any two points are the extremities of a unique arc in $%
X$ (cf. \cite{dm}). Let $\Gamma $ be a lattice in a semisimple algebraic
group of rank at least two. Buchesne and Monod \cite{dm} prove that any
action of $\Gamma $ on a dendrite $X$ fixes a fixed point or a pair of two
points. An $\mathbb{R}$--tree is a geodesic metric space in which there is a
unique arc connecting each pair of points. Bogopolski \cite{bo} prove that
any isometric action of $\mathrm{Aut}(F_{n})$ on a simplicial tree has a
fixed point. Culler and Vogtmann \cite{cv} give a short proof based on their
idea of \textquotedblleft minipotent\textquotedblright\ elements. Bridson 
\cite{brid08} proves a similar result for group actions on $\mathbb{R}$%
-trees using the triangle condition. We prove the following.

\begin{theorem}
\label{last}Let $X$ be a unique arcwise connected space (eg. a tree or
dendrite). Any action of $\mathrm{Aut}(F_{n})$ (or $\mathrm{SL}_{n}(\mathbb{Z%
}),n\geq 3$) on $X$ by homeomorphisms has a fixed point.
\end{theorem}

This gives a simultaneous generalization of the fixed point property for
both group actions on simplicial trees, $\mathbb{R}$-trees and those on
dendrites.

\section{Notations and basic facts}

\subsection{\textrm{CAT(0)} spaces}

Let $(X,d_{X})$ be a geodesic metric space. For three points $x,y,z\in X,$
the geodesic triangle $\Delta (x,y,z)$ consists of the three vertices $x,y,z$
and the three geodesics $[x,y],[y,z]$ and $[z,x].$ Let $\mathbb{R}^{2}$ be
the Euclidean plane with the standard distance $d_{\mathbb{R}^{2}}$ and $%
\bar{\Delta}$ a triangle in $\mathbb{R}^{2}$ with the same edge lengths as $%
\Delta $. Denote by $\varphi :\Delta \rightarrow \bar{\Delta}$ the map
sending each edge of $\Delta $ to the corresponding edge of $\bar{\Delta}.$
The space $X$ is called a \textrm{CAT(0)} space if for any triangle $\Delta $
and two elements $a,b\in \Delta ,$ we have the inequality 
\begin{equation*}
d_{X}(a,b)\leq d_{\mathbb{R}^{2}}(\varphi (a),\varphi (b)).
\end{equation*}%
The typical examples of \textrm{CAT(0)} spaces include simplicial trees,
hyperbolic spaces, products of \textrm{CAT(0) }spaces and so on. From now
on, we assume that $X$ is a complete \textrm{CAT(0)} space. Denote by 
\textrm{Isom}$(X)$ the isometry group of $X.$ For any $g\in $ \textrm{Isom}$%
(X)$\textrm{, }let 
\begin{equation*}
\mathrm{Minset}(g)=\{x\in X:d(x,gx)\leq d(y,gy)\text{ for any }y\in X\}
\end{equation*}%
and let $\tau (g)=\inf\nolimits_{x\in X}d(x,gx)$ be the translation length
of $g.$ When the fixed-point set $\mathrm{Fix}(g)\neq \emptyset ,$ we call $%
g $ elliptic. When $\mathrm{Minset}(g)\neq \emptyset $ and $d_{X}(x,gx)=\tau
(g)>0$ for any $x\in \mathrm{Minset}(g),$ we call $g$ hyperbolic. The group
element $g$ is called semisimple if the minimal set $\mathrm{Minset}(g)$ is
not empty, i.e. it is either elliptic or hyperbolic. A subset $C$ of a 
\textrm{CAT(0)} space if convex, if any two points $x,y\in C$ can connected
by the geodesic segment $[x,y]\subset C.$ For more details on \textrm{CAT(0) 
}spaces, see the book of Bridson and Haefliger \cite{bh}.

The following lemma is from \cite{bh} (II.2.4).

\begin{lemma}
\label{proj}Let $\gamma :X\rightarrow X$ be an isometry of a CAT(0) space $%
X\ $and $C$ an $\gamma $-invariant convex, complete subspace$.$

(1) For any $x\in X,$ there exists a unique $p(x)\in C$ such that $%
d(x,p(x))=d(x,C):=\inf_{c\in C}d(x,c).$ This gives a projection $%
p:X\rightarrow C,$ which is distance-non-increasing.

(2) We have $p(\gamma x)=\gamma p(x)$ for any $x\in X.$ Moreover, $\mathrm{%
Min}(\gamma |_{C})=\mathrm{Min}(\gamma )\cap C.$
\end{lemma}

\begin{proof}
The first part (1) is \cite{bh} (II.2.4). The second part is in the proof
II.6.2(4) of \cite{bh}. Actually, we have that 
\begin{equation*}
d(\gamma x,\gamma px)=d(x,px)=d(x,C)=d(\gamma x,C)
\end{equation*}
and thus $\gamma px=p(\gamma x)$ by the uniqueness of projection points.
Therefore, $p\mathrm{Min}(X)=\mathrm{Min}(\gamma )\cap C=\mathrm{Min}(\gamma
|_{C}).$
\end{proof}

\begin{lemma}
\label{le1}Let $G=A\rtimes H$ be a semi-direct product of two groups $A$ and 
$H.$ Suppose that both $A$ and $H$ have non-empty fixed point sets $X^{A}$
and $X^{H}$. Then $G$ has non-empty fixed point set $X^{G}.$
\end{lemma}

\begin{proof}
For any $x\in X^{A}$ and any $h\in H,a\in A,$ we have that 
\begin{equation*}
ahx=h(h^{-1}ah)x=hx.
\end{equation*}
This shows that $hX^{A}=X^{A}.$ Moreover, the fixed point $X^{A}$ is a
convex closed complete subspace of $X$. Let $p:X\rightarrow X^{A}$ be the
projection as defined in Lemma \ref{proj}. Since $p(hx)=hp(x)$ for any $x\in
X,$ we have that $p(X^{H})=X^{H}\cap X^{A}=X^{G}\neq \emptyset .$
\end{proof}

\subsection{Matrix groups\label{elem}}

In this subsection, we briefly recall the definitions of the elementary
subgroups $E_{n}(R)$ of the general linear group $\mathrm{GL}_{n}(R)$. Let $%
R $ be an associative ring with identity and $n\geq 2$ be an integer. The
general linear group $\mathrm{GL}_{n}(R)$ is the group of all $n\times n$
invertible matrices with entries in $R$. For an element $r\in R$ and any
integers $i,j$ such that $1\leq i\neq j\leq n,$ denote by $e_{ij}(r)$ the
elementary $n\times n$ matrix with $1s$ in the diagonal positions and $r$ in
the $(i,j)$-th position and zeros elsewhere. The group $E_{n}(R)$ is
generated by all such $e_{ij}(r),$\textsl{\ i.e. }%
\begin{equation*}
E_{n}(R)=\langle e_{ij}(r)|1\leq i\neq j\leq n,r\in R\rangle .
\end{equation*}%
Let $D=\{\mathrm{diag}(\varepsilon _{1},\varepsilon _{2},\cdots ,\varepsilon
_{n})\mid \varepsilon _{i}=\pm 1\}$ be the diagonal subgroup. The extended
elementary group $E_{n}^{\prime }(R)$ is defined as $E_{n}(R),$ when $n$ is
odd and $\langle E_{n}(R),D\rangle <\mathrm{GL}_{n}(R)$ when $n$ is even.
Note that%
\begin{equation*}
\begin{pmatrix}
-1 &  \\ 
& -1%
\end{pmatrix}%
=%
\begin{pmatrix}
1 & 1 \\ 
& 1%
\end{pmatrix}%
\begin{pmatrix}
1 &  \\ 
-1 & 1%
\end{pmatrix}%
\begin{pmatrix}
1 & 2 \\ 
& 1%
\end{pmatrix}%
\begin{pmatrix}
1 &  \\ 
-1 & 1%
\end{pmatrix}%
\begin{pmatrix}
1 & 1 \\ 
& 1%
\end{pmatrix}%
.
\end{equation*}%
This implies that $\langle E_{n}(R),D\rangle $ contains $E_{n}(R)$ as an
index-2 subgroup. When $R=\mathbb{Z}$ and $n$ is even, we have $%
E_{n}^{\prime }(\mathbb{Z})=\mathrm{GL}_{n}(\mathbb{Z}).$

Denote by $I_{n}$ the identity matrix and by $[a,b]$ the commutator $%
aba^{-1}b^{-1}.$ The following lemma displays the commutator formulas for $%
E_{n}(R)$ (cf. Lemma 9.4 in \cite{mag}).

\begin{lemma}
\label{ecom}Let $R$ be a ring and $r,s\in R.$ Then for distinct integers $%
i,j,k,l$ with $1\leq i,j,k,l\leq n,$ the following hold:

\begin{enumerate}
\item[(1)] $e_{ij}(r+s)=e_{ij}(r)e_{ij}(s);$

\item[(2)] $[e_{ij}(r),e_{jk}(s)]=e_{ik}(rs);$

\item[(3)] $[e_{ij}(r),e_{kl}(s)]=I_{n}.$
\end{enumerate}
\end{lemma}

By Lemma \ref{ecom}, the group $E_{n}(R)$ $(n\geq 3)$ is finitely generated
when the ring $R$ is finitely generated.

\section{Basic facts on the relative property $\mathcal{FA}_{n}$}

\begin{definition}
Let $H$ be a subgroup of a group $G.$ The pair $(G,H)$ has the relative
property $\mathcal{FA}_{n}$ (resp. $s\mathcal{FA}_{n}$) if whenever $G$ acts
isometrically (resp. semi-simply) on a complete $n$-dimensional $\mathrm{CAT}%
(0)$ space $X,$ the subgroup $H$ has a fixed point.
\end{definition}

We call that $G$ has property $\mathcal{FA}_{n}$ if $(G,G)$ has the relative
property $\mathcal{FA}_{n}.$ The following are some basic properties.

\begin{lemma}
\label{basic}(1) For groups $H<K<G,$ if $(K,H)$ has $\mathcal{FA}_{n}$
(resp. $s\mathcal{FA}_{n}$), then so does $(G,H).$ If $(G,K)$ has $\mathcal{%
FA}_{n}$ (resp. $s\mathcal{FA}_{n}$), then so does $(G,H).$

(2) Let $H_{2}<H_{1}$ be two subgroups of $G.$ If $H_{1}$ is bounded
generated by finitely many conjugates of $H_{2},$ then $(G,H_{1})$ has $%
\mathcal{FA}_{n}$ (resp. $s\mathcal{FA}_{n}$) if and only if $(G,H_{2})$ has 
$\mathcal{FA}_{n}$ (resp. $s\mathcal{FA}_{n}$).

(3) If $H$ is of finite index in $G,$ then $(G,H)$ has $\mathcal{FA}_{n}$
(resp. $s\mathcal{FA}_{n}$) if and only if $G$ has $\mathcal{FA}_{n}$ (resp. 
$s\mathcal{FA}_{n}$).

(4) If $(G,H)$ has the relative property $\mathcal{FA}_{n+1}$ (resp. $s%
\mathcal{FA}_{n}$), then $(G,H)$ has the relative property $\mathcal{FA}_{n}$
(resp. $s\mathcal{FA}_{n}$).
\end{lemma}

\begin{proof}
(1) is obvious from the definition.

For (2), let $H_{1}$ be bounded generated by $K_{1},K_{2},\cdots ,K_{k}$
(i.e. there exists $N>0$ such that each element $g\in G$ is a product of at
most $N$ elements in $\cup _{i=1}^{k}K_{i}$), where each $K_{i}$ is a
conjugation of $H_{2}.$ When $H_{2}$ has a fixed point $x$, each $K_{i}$ has
a fixed point. Since each orbit $H_{i}x$ is bounded, the orbit $H_{1}x$ is
bounded and thus $H_{1}$ has a fixed point. When $(G,H_{1})$ has $\mathcal{FA%
}_{n},$ it follows (1) that $(G,H_{2})$ has $\mathcal{FA}_{n}$.

For (3), when $H$ has a fixed point $x,$ the orbit $Gx$ is bounded and thus $%
G$ also has a fixed point.

For (4), when $G$ acts on a $n$-dimensional complete $\mathrm{CAT}(0)$ space 
$X$ by isometries, the group $G$ can also act on the product $X\times 
\mathbb{R}$ by trivial action on the second component. When $H$ has a fixed
point in $X\times \mathbb{R},$ it has also a fixed point in $X.$
\end{proof}

\begin{lemma}
Let $H$ be a nilpotent group subgroup of a finitely generated group $G.$
Suppose that each element in $H$ is distorted in $G.$ Then $(G,H)$ has the
relative property $s\mathcal{FA}_{n}$ for any $n.$
\end{lemma}

\begin{proof}
When $G$ acts on a complete $\mathrm{CAT}(0)$ space $X$ by semisimple
isometries, each element in $H$ is either hyperbolic or elliptic. Suppose
that some element $h\in H$ is hyperbolic and thus has a translation axis $l$%
. Fix a finite generating set $S<G$ and denote by $l_{S}$ the word length
function. For $x\in l,$ the translation 
\begin{eqnarray*}
|h| &=&d(x,hx)=\lim_{n\rightarrow \infty }\frac{d(x,h^{n}x)}{n} \\
&\leq &\lim_{n\rightarrow \infty }\frac{l_{s}(h^{n})Max\{d(x,sx):s\in S\}}{n}%
=0,
\end{eqnarray*}%
which is impossible. This shows that each element in $H$ is elliptic. Since $%
H$ is nilpotent, there is an upper central series%
\begin{equation*}
\{1\}=Z_{0}\trianglelefteq Z_{1}\trianglelefteq \cdots \trianglelefteq
Z_{n}=H
\end{equation*}%
such that each $Z_{i+1}/Z_{i}\ $is the center of $G/Z_{i}.$ Note that $Z_{1}$
is abelian. Any finitely generated subgroup $K$ of $Z_{1}$ has a fixed point
by Lemma \ref{le1}. Since the countable abelian group $Z_{1}$ is a direct
limit of finitely generated subgroups, it also has a fixed point.
Inductively suppose that $Z_{i}$ has a fixed point. The quotient (abelian)
group $Z_{i+1}/Z_{i}$ acts invariantly on the fixed point set $\mathrm{Fix}%
(Z_{i})$ and also thus has a fixed point. This proves that $Z_{i+1}$ and
thus $H$ has a fixed point.
\end{proof}

\begin{remark}
The previous lemma does not hold for $\mathcal{FA}_{n}.$ For example, the
subgroup $\langle e_{12}(1)\rangle \cong \mathbb{Z}$ generated by the $%
n\times n$ matrix with $1$ in the $(1,1)$-th position and the diagonals, $0$%
s elsewhere, is distorted in $\mathrm{SL}_{n}(\mathbb{Z})$ $(n\geq 3).$ But $%
\mathrm{SL}_{n}(\mathbb{Z})$ acts properly on the symmetric space $\mathrm{SL%
}_{n}(\mathbb{R})/SO(n).$
\end{remark}

\section{Helly's theorem and Coxeter groups}

The classical Helly's theorem says that for $m$ non-empty convex subsets $%
\{F_{i}\}_{i=1}^{m}$ $(m\geq n+2)$ in the Euclidean space $\mathbb{R}^{n}$
satisfying the intersection of any $n+1$ sets $\cap _{j=1}^{n}F_{i_{j}}\neq
\emptyset ,$ the whole intersection $\cap _{i=1}^{m}F_{i}\neq \emptyset .$
The following is the Helly theorem for $\mathrm{CAT}(0)$ spaces (cf. Bridson 
\cite{brid}, Theorem 3.2; Varghese \cite{var}, Theorem 5.10.)

\begin{lemma}
\label{helly}(the Helly theorem) Let $X$ be a $d$-dimensional (i.e. covering
dimension) complete $\mathrm{CAT}(0)$ space and $S$ a finite family of
non-empty closed convex subspaces. If the intersection of each $(d+1)$%
-elements of $S$ is non-empty, then $\cap S$ is non-empty.
\end{lemma}

\begin{corollary}
(\cite{var}, 5.1.) Let $G$ be a group, $Y$ a finite generating set of $G$
and $X$ $d$-dimensional complete $\mathrm{CAT}(0)$ space. Let $\Phi
:G\rightarrow \mathrm{Isom}(X)$ be a homomorphism. If each $(d+1)$-element
subset of $Y$ has a fixed point in $X$, then $G$ has a fixed point in $X$.
\end{corollary}

The following was firstly proved by Bridson \cite{brid} (Proposition 3.4):

\begin{lemma}
\label{3.4}Let $X^{d}$ be a complete $\mathrm{CAT}(0)$ space and $G$ a group
acting on $X$ by isometries. Suppose that there are integers $%
k_{1},k_{2},...,k_{m}$ and subset $S_{i},i=1,2,...,m$ satisfying

(1) every $k_{i}$-element of $S_{i}$ has a common fixed point;

(2) $[s_{i},s_{j}]=1$ for any $s_{i}\in S_{i},s_{j}\in S_{j},i\neq j;$

(3) $d<k_{1}+k_{2}+\cdots +k_{m}.$

Then for some $i$ every finitely many elements in $S_{i}$ have a common
fixed point.
\end{lemma}

\subsection{The action of Coxeter groups}

A Coxeter group $(W,S)$ is a group $W$ with the presentation%
\begin{equation*}
\langle s_{1},s_{2},\cdots ,s_{n}\mid (s_{i}s_{j})^{m_{ij}}=1\rangle
\end{equation*}%
where $S=\{s_{1},s_{2},\cdots ,s_{n}\},$ $m_{ii}=1$ and $m_{ij}\geq 2$ for $%
i\neq j.$ It is possible that $m_{ij}=\infty ,$ which means no relations of
the form $(s_{i}s_{j})^{\infty }$. There is a canonical group homomorphism $%
\rho :W\rightarrow \mathbb{Z}/2$ defined by mapping each $s_{i}$ to the
generator $1\in \mathbb{Z}/2.$ The kernel is denoted by $SW=\ker \rho ,$
called the alternating Coxeter subgroup. For a subset $T\subseteq S$, the
subgroup $W_{T}<W$ generated by elements in $T$ is called a parabolic
subgroup (or a special subgroup in the literature). The Coxeter matrix is
the matrix $M=(m_{ij})_{n\times n}$ and the Schl\"{a}fli matrix is $C=(-\cos
(\frac{\pi }{m_{ij}}))_{n\times n}.$ The Coxeter graph of $(W,S)$ is a graph
defined as the following. The vertices of the graph are labelled by
generators in $S.$ Vertices $(s_{i},s_{j})$ are adjacent if and only if $%
m_{ij}\geq 3$. An edge is labelled with the value of $m_{ij}$ whenever the
value is $\geq 4.$ The following is well-known (cf. \cite{Davis}).

\begin{lemma}
\label{finite}The Coxeter group is finite if and only the eigenvalues of the
Schl\"{a}fli matrix are all positive. The finite Coxeter groups consist of
three one-parameter families of increasing rank $A_{n},B_{n},D_{n}$, one
one-parameter family of dimension two $I_{2}(p)$, and six exceptional
groups: $E_{6},E_{7},E_{8},F_{4},H_{3}$ and $H_{4}.$
\end{lemma}

\bigskip 
\begin{equation*}
\begin{array}{c}
\FRAME{itbpF}{4.2531in}{1.1018in}{0in}{}{}{Figure}{\special{language
"Scientific Word";type "GRAPHIC";maintain-aspect-ratio TRUE;display
"USEDEF";valid_file "T";width 4.2531in;height 1.1018in;depth
0in;original-width 9.5692in;original-height 2.4587in;cropleft "0";croptop
"1";cropright "1";cropbottom "0";tempfilename
'PR7VXG02.bmp';tempfile-properties "XPR";}} \\ 
\text{Figure 1: Coxeter graphs of the finite Coxeter groups}%
\end{array}%
\end{equation*}

\begin{definition}
Let $n>0$ be an integer. A Coxeter group $(W,S)$ is $n$-spherical if any $%
n+1 $ elements in $S$ generates a finite group.
\end{definition}

\begin{lemma}
\label{coxe}Let $X$ be a complete $\mathrm{CAT}(0)$ space of dimension $%
n\geq 1$. Let $(W,S)$ be an $n$-spherical Coxeter group. Then any action of $%
W$ on $X$ has a fixed point.
\end{lemma}

\begin{proof}
By the assumption, the fixed point set $X^{s}$ is closed for each generator $%
s\in S$. The intersection $\cap
_{i=1}^{k}X^{s_{i}}=X^{W_{\{s_{1},s_{2},\cdots ,s_{k}\}}}$ is not empty for $%
k\leq n+1,$ since $W_{\{s_{1},s_{2},\cdots ,s_{k}\}}$ is finite. The Helly
theorem (cf. Lemma \ref{helly}) implies that the fixed point set $X^{W}$ is
not empty.
\end{proof}

\section{Actions of matrix groups on $\mathrm{CAT}(0)$ spaces}

In this section, we always suppose that $X$ is a complete $\mathrm{CAT}(0)$
space and $\mathrm{Isom}(X)$ is the group of isometries. Let $R$ be a
finitely generated ring and $E_{n}(R)$ (resp. $E_{n}^{\prime }(R)$) the
(resp. extended) elementary subgroup. In this section, we will prove Theorem %
\ref{1.1}.

\begin{lemma}
\label{lem1.2}Let $1\leq j\neq j\leq n.$ The elementary subgroup $E_{n}(R)$
has property $\mathcal{FA}_{n-2}$ if and only if $(E_{n}(R),e_{ij}(x))$ has
property $\mathcal{FA}_{n-2}$ for any $x\in R.$
\end{lemma}

\begin{proof}
For any $y\in R,$ the two matrices $e_{ij}(x)$ and $e_{ij}(y)$ commute with
each other. Lemma \ref{le1} implies that they have a common fixed point.
Therefore, the subgroup $e_{ij}:=\langle e_{ij}(x):x\in R\rangle $ has a
fixed point. Since all $e_{ij}$ $(i\neq j$) are conjugate, $e_{12}$ has a
fixed point. Apply Lemma \ref{le1} again to get that the subgroup $\langle
e_{1i}(x)\mid 2\leq i\leq n,x\in R\rangle $ has a fixed point and the upper
triangular subgroup $H=\langle e_{ij}(x)\mid 1\leq i<j\leq n,x\in R\rangle $
has a fixed point. Note that $E_{n}(R)$ is generated by $S=\{e_{12},e_{23},%
\cdots ,e_{n-1,n},e_{n,1}\}$ and any $n-1$ elements of $S$ generate a
subgroup isomorphic to $H.$ The Helly theorem implies that $E_{n}(R)$ has a
global fixed point.
\end{proof}

\bigskip

Denote by $\{a_{i}\}_{i=1}^{n}$ the standard basis of the finitely generated
free abelian group $\mathbb{Z}^{n},$ by $\{\sigma _{i}\}_{i=1}^{n}$ the
standard basis of the elementary abelian group $(\mathbb{Z}/2)^{n}.$ Let $G=%
\mathbb{Z}^{n}\rtimes ((\mathbb{Z}/2)^{n}\rtimes S_{n})$ be the semi-direct
product, where $\mathbb{Z}/2$ acts on $\mathbb{Z}$ by reflection and $S_{n}$
acts on $(\mathbb{Z}/2)^{n}$ and $\mathbb{Z}^{n}$ by permuting the standard
basis. Let $\rho :(\mathbb{Z}/2)^{n}\rtimes S_{n}\rightarrow \{+1,-1\}$ be
the group homomorphism mapping each $\sigma _{i}$ and each permutation $(ij)$
to $1$. Denote by $A_{n}^{+}=\ker \rho $ and $G^{+}=\mathbb{Z}^{n}\rtimes
A_{n}^{+}.$

\begin{proposition}
\label{prop} Let $n\geq 2.$ Any isometric action of $G$ on a complete $%
\mathrm{CAT}(0)$ space $X^{d}$ $(d<n)$ has a global fixed point. In other
words, the group $G$ has the property $\mathcal{FA}_{n-1}.$
\end{proposition}

\begin{proof}
Choose 
\begin{equation*}
S=\{a_{1}\sigma _{1},(12),(23),...,(n-1,n),\sigma _{n}\}.
\end{equation*}%
Note that each element of $S$ is of order two. Let $H=(W,S)$ be the coxeter
group assigned to the Coxeter matrix $(m_{ij})$ defined by $%
m_{ii}=2,m_{12}=4,m_{i,i+1}=3$ for $2\leq i\leq n$ and $m_{n,n+1}=4$ and all
other entries $m_{ij}=\infty .$ Recall that 
\begin{equation*}
H=\langle s_{i}\in S\mid (s_{i}s_{j})^{m_{ij}}=1,1\leq i,j\leq n\rangle .
\end{equation*}%
Let $f:H\rightarrow \langle S\rangle $ be defined as%
\begin{eqnarray*}
s_{1} &\rightarrow &a_{1}\sigma _{1}, \\
s_{i+1} &\rightarrow &(i,i+1),1\leq i\leq n-1, \\
s_{n+1} &\rightarrow &\sigma _{n}.
\end{eqnarray*}%
It could be directly checked that $f$ is a group homomorphism. The Coxeter
graph of $(W,S)$ is a path consisting of $n+1$ vertices with the edges $%
(s_{1},s_{2}),(s_{n},s_{n+1})$ labled by $4.$ The subgraph spanned by any $n$
vertices is either a path of type $B_{n}$ or disjoint union of two paths of
type $B_{k}$. By the classification of finite Coxeter groups (cf. Lemma \ref%
{finite}), any spherical subgroup of rank $n$ is finite in $H$. This proves
that the subgroup generated by any $n$ elements is finite. The Helly theorem
(cf. Lemma \ref{helly}) implies that $S$ has a common fixed point $x$. Note
that $a_{1}\in \langle S\rangle $ and thus $\mathbb{Z}^{n}<\langle S\rangle
. $ Therefore, the action of $G$ on $X$ has a bounded orbit $Gx$ and thus
has a global fixed point by Lemma \ref{basic} (3).
\end{proof}

\bigskip

Note that $\mathbb{Z}^{n}\rtimes ((\mathbb{Z}/2)^{n}\rtimes S_{n})$ could
act on the Euclidean space $\mathbb{R}^{n}$ in the standard way. This shows
that the bound in the previous proposition is sharp.

\bigskip

\begin{lemma}
\label{lem1.1}Let $E_{n}(R)$ act isometrically on a complete $\mathrm{CAT}%
(0) $ space $X^{d}$ $(d<n-1).$ Assume $n$ is odd. For any element $x\in R,$
the subgroup $\langle e_{1i}(x)\mid i=2,3,\cdots ,n\rangle $ has a fixed
point.
\end{lemma}

\begin{proof}
When $n$ is odd, let $H$ be the subgroup generated by 
\begin{equation*}
\{-I_{n}(12),-I_{n}(23),...,-I_{n}(n-1,n),-I_{n}\sigma _{n}\}
\end{equation*}%
and $\langle e_{1i}(x)\mid i=2,3,\cdots ,n\rangle .$ Define a group
homomorphism%
\begin{equation*}
\alpha :\mathbb{Z}^{n}\rtimes ((\mathbb{Z}/2)^{n}\rtimes S_{n})\rightarrow H
\end{equation*}%
by 
\begin{equation*}
a_{i}\mapsto e_{1,i+1}(x),\sigma _{i}\mapsto -I_{n}\varepsilon _{i}
\end{equation*}%
and 
\begin{equation*}
(i,i+1)\mapsto -I_{n}(i,i+1).
\end{equation*}%
Proposition \ref{prop} implies that $H$ has a fixed point.
\end{proof}

\bigskip

\bigskip

\begin{proof}[Proof of Theorem \protect\ref{1.2}]
Suppose that $R^{n}\rtimes E_{n}^{\prime }(R)$ acts on a complete $n-1$
dimensional $\mathrm{CAT}(0)$ space $X.$ When $n$ is odd, we view $(\mathbb{Z%
}/2)^{n}\rtimes S_{n}$ as a subgroup of $\mathrm{SL}_{n}(\mathbb{Z})=E_{n}(%
\mathbb{Z})$ through the group homomorphism $\alpha $ in the proof of Lemma %
\ref{lem1.1}. The ring homomorphism $\mathbb{Z}\rightarrow R$ preserving the
identity induces a group homomorphism $\mathrm{GL}_{n}(\mathbb{Z}%
)\rightarrow E_{n}^{\prime }(R).$ Without confusions, we still denote the
image by $(\mathbb{Z}/2)^{n}\rtimes S_{n}.$ For any $x\in R,$ let 
\begin{eqnarray*}
S_{x} &=&\{(x,\sigma _{1}),(12),(23),...,(n-1,n),\sigma _{n}\} \\
&\subseteq &R^{n}\rtimes E_{n}^{\prime }(R).
\end{eqnarray*}%
The obvious map $S\rightarrow S_{x}$ by $\sigma _{1}a_{1}\rightarrow
((x,0,\cdots ,0),\sigma _{1})$ induces a group homorphism%
\begin{equation*}
\mathbb{Z}^{n}\rtimes ((\mathbb{Z}/2)^{n}\rtimes S_{n})\rightarrow
R^{n}\rtimes ((\mathbb{Z}/2)^{n}\rtimes S_{n}).
\end{equation*}%
Proposition \ref{prop} implies that the subgroup $xR^{n}$ (for any $x$) and
thus $R^{n}$ has a fixed point.
\end{proof}

\bigskip

\begin{proof}[Proof of Theorem \protect\ref{1.1}]
When $n$ is odd, we have $E_{n}^{\prime }(R)=E_{n}(R).$ Theorem \ref{1.1}
follows Lemma \ref{lem1.1} and Lemma \ref{lem1.2}. When $n$ is even, the
same group homomoprhism $\alpha $ as in Lemma \ref{lem1.1} and Proposition %
\ref{prop} imply that the subgroup $\langle e_{1i}(x)\mid i=2,3,\cdots
,n\rangle $ has a fixed point. Theorem \ref{1.1} is proved by Lemma \ref{1.2}%
.
\end{proof}

\bigskip

\section{Actions of $\mathrm{Aut}(F_{n})$ on $\mathrm{CAT}(0)$ spaces}

Suppose that the free $F_{n}=\langle a_{1},...,a_{n}\rangle $ and $\mathrm{%
Aut}(F_{n})$ the group of automorphisms. For any $1\leq i\leq n,$ let $%
\varepsilon _{i}:F_{n}\rightarrow F_{n}$ be the involution defined by $%
a_{i}\rightarrow a_{i}^{-1}$ and $a_{j}\rightarrow a_{j}$ for any $j\neq i.$
The Nielsen transformations are defined by 
\begin{equation*}
\rho _{ij}:a_{i}\rightarrow a_{i}a_{j},a_{k}\rightarrow a_{k},k\neq i
\end{equation*}%
and 
\begin{equation*}
\lambda _{ij}:a_{i}\rightarrow a_{j}a_{i},a_{k}\rightarrow a_{k},k\neq i.
\end{equation*}%
For $1\leq i\neq j\leq n,$ denote by $(ij)$ the permutation of basis. It is
well-known that $\mathrm{Aut}(F_{n})$ is generated by the elements $\rho
_{ij},\lambda _{ij}$ and $\varepsilon _{i}$ (see Gersten \cite{ger}).

The following lemma is a variation on the argument that Gersten \cite{ger12}
used to show that $\mathrm{Aut}(F_{n})$ cannot act properly and cocompactly
on a $\mathrm{CAT}(0)$ space if $n\geq 3.$ (cf. \cite{bh} p. 253).

\begin{lemma}
\label{gers}Let $G=\langle \alpha ,\beta ,t\mid t^{p}\beta t^{-p}=\beta
\alpha ^{p}$ for any integer $p\rangle .$ If $(G,\langle t\rangle )$ has the
relative property $\mathcal{FA}_{n},$ then $(G,\langle \alpha \rangle )$ has 
$\mathcal{FA}_{n}$ for any $n.$
\end{lemma}

\begin{proof}
Suppose that $G$ acts on an $n$-dimensional complete $\mathrm{CAT}(0)$ space 
$X$ and $t$ has a fixed point $x\in X$. For any $p$ we have%
\begin{eqnarray*}
d(\alpha ^{p}x,x) &=&d(\beta ^{-1}t^{p}\beta t^{-p}x,x)\leq d(\beta
^{-1}t^{p}\beta x,\beta ^{-1}t^{p}x)+d(\beta ^{-1}t^{p}x,x) \\
&\leq &d(\beta x,x)+d(\beta ^{-1}x,x).
\end{eqnarray*}%
Therefore, the orbit $\{\alpha ^{p}x:p\in \mathbb{Z}\}$ is bounded and thus $%
\alpha $ has a fixed point.
\end{proof}

\begin{theorem}
\label{niel}Let $X^{d}$ be a complete $\mathrm{CAT}(0)$ space and $\mathrm{%
Aut}(F_{n})$ $(n\geq 4)$ act on $X$ by isometries. When $d<n-2,$ each
Nielsen transformation of $\mathrm{Aut}(F_{n})$ has a fixed point.
\end{theorem}

\begin{proof}
Let $H$ be the subgroup generated by 
\begin{equation*}
S=\{\rho _{21}^{-1}\rho _{31}(23),(34),...,(n-1,n),(2,n)\}.
\end{equation*}%
It can be directly checked that the product of any succesive two elements
(including that of the first element and the last one) is of order $3.$ Let $%
H$ be the coxeter group assigned to the Coxeter graph $K$ a loop of length $%
n-1.$ Explicitly, the group 
\begin{equation*}
H=\langle s_{1},\cdots ,s_{n-1}\mid s_{i}^{2}=(s_{i}s_{i+1})^{3}=1,1\leq
i\leq n-1\rangle ,
\end{equation*}%
where $s_{n}=s_{1}.$ Let $f:H\rightarrow \langle S\rangle $ be defined as 
\begin{eqnarray*}
s_{1} &\rightarrow &\rho _{21}^{-1}\rho _{31}(23), \\
s_{i} &\rightarrow &(i+1,i+2),2\leq i\leq n-2, \\
s_{n-1} &\rightarrow &(2,n).
\end{eqnarray*}

It can be directly checked that $f$ is a group homomorphism. Since the
subgraph spanned by $n-2$ vertices in the Coxter graph $K$ is of type $A,$
any $n-2$ elements in $S$ generate a finite group (cf. Lemma \ref{finite}).
By the Helley theorem, $S$ has a common fixed point. Therefore, the element $%
\rho _{21}^{-1}\rho _{31}\in \langle S\rangle $ has a fixed point. Thus $%
\rho _{ij}^{-1}\rho _{i^{\prime }j}$ has fixed point for any $i,i^{\prime
},j $ after conjugating by a permutation.

Let $T=\rho _{32}^{-1}\rho _{42}.$ It can be directly checked that $%
T^{-p}\rho _{13}T^{p}=\rho _{13}\rho _{12}^{p}$ for any integer $p.$ Let $%
G=\langle \alpha ,\beta ,t\mid t^{p}\beta t^{-p}=\beta \alpha ^{p}$ for any
integer $p\rangle .$ Define a group homomorphism%
\begin{equation*}
f:G\rightarrow \mathrm{Aut}(F_{n})
\end{equation*}%
by $f(t)=T,$ $f(\alpha )=\rho _{12},f(\beta )=\rho _{13}.$ Lemma \ref{gers}
implies that $\rho _{12}$ has a fixed point. Since any Nielsen
transformation is conjugate to $\rho _{12},$ the statement is proved.
\end{proof}

\bigskip

Theorem \ref{1.3} is the same as the following theorem.

\begin{theorem}
Let $X^{d}$ be a complete $\mathrm{CAT}(0)$ space and $\mathrm{Aut}(F_{n})$ $%
(n\geq 4)$ act on $X$ by isometries. When $d<n-2,$ the subgroup $\mathrm{Aut}%
(F_{2})$ has a fixed point.
\end{theorem}

\begin{proof}
It is already proved in Theorem \ref{niel} that each Nielsen transformation
has a fixed point. Choose 
\begin{equation*}
S=\{\rho _{12}\varepsilon _{2},\varepsilon _{1},\varepsilon _{2},\varepsilon
_{1}\varepsilon _{2}(12)\}.
\end{equation*}%
We check that every two elements on $S$ has a common fixed in the following:

The elements $\rho _{12}\varepsilon _{2}$ and $\varepsilon _{1}$ lie in $%
\langle \rho _{12},\lambda _{12}\rangle \rtimes \langle \varepsilon
_{2},\varepsilon _{1}\rangle \cong \mathbb{Z}^{2}\rtimes (\mathbb{Z}/2)^{2}.$
Since $\rho _{12}$ commutes with $\lambda _{12},$ they have a common fixed
point. Therefore, the elements $\rho _{12}\varepsilon _{2}$ and $\varepsilon
_{1}$ have a common fixed point. The elements $\rho _{12}\varepsilon _{2}$
and $\varepsilon _{2}$ generate an infinite Dihedral group and thus have a
common fixed point. The elements $\rho _{12}\varepsilon _{2}$ and $%
\varepsilon _{1}\varepsilon _{2}(12)$ generate a finite group and thus have
a common fixed point.

Note that $\mathrm{Aut}(F_{n})$ contains $2[\frac{n}{2}]$ copies of $\mathrm{%
Aut}(F_{2})$ along the diagonal. The Bridson's lemma \ref{3.4} implies that
when $d<2[\frac{n}{2}],$ the set $S$ has a common fixed point. Since $S$
generates $\mathrm{Aut}(F_{2}),$ the proof is finished.
\end{proof}

\begin{lemma}
Let $X^{d}$ be a complete $\mathrm{CAT}(0)$ space and $\mathrm{Aut}(F_{n})$ $%
(n\geq 5)$ act on $X$ by isometries. When $d<n-3,$ the subgroup $\mathrm{Aut}%
(F_{2})\ltimes F_{2}$ has a fixed point.
\end{lemma}

\begin{proof}
Let $T:F_{n}\rightarrow F_{n}$ be the transformation given by 
\begin{equation*}
x_{4}\rightarrow x_{4}w^{-1},x_{5}\rightarrow x_{5}w,
\end{equation*}%
for a word $w\in \langle x_{1},x_{3}\rangle <F_{n}.$ Then $T^{-1}\rho
_{24}T=\rho _{24}\rho _{2,w},$ where $\rho _{2,w}(x_{2})=x_{2}w$ and $\rho
_{2,w}(x_{i})=x_{i}$ for any $i\neq 2.$ Consider the set 
\begin{equation*}
S=\{T(45),(56),...,(n-1,n),(2,n),(4,n)\}.
\end{equation*}%
A similar argument using Coxeter groups as in the proof of Theorem \ref{niel}
shows that any $n-4$ elements of $S$ generate a finite group. This implies
that when $d<n-3,$ the whole group has a fixed point by the Helly's theorem.
Therefore, the element $T$ has a fixed point. Note that for any integer $p,$
we have $T^{-p}\rho _{24}T^{p}=\rho _{24}\rho _{2,w}^{p}.$ Let $G=\langle
\alpha ,\beta ,t\mid t^{p}\beta t^{-p}=\beta \alpha ^{p}$ for any integer $%
p\rangle .$ Define a group homomorphism%
\begin{equation*}
f:G\rightarrow \mathrm{Aut}(F_{n})
\end{equation*}%
by $f(t)=T,$ $f(\alpha )=\rho _{2,w},f(\beta )=\rho _{24}.$ Lemma \ref{gers}
implies that $\rho _{2,w}$ has a fixed point. Since $(12)\rho
_{2,w}(12)=\rho _{1,v},v\in \langle x_{2},x_{3}\rangle ,$ we see that $%
F_{2}=\langle \rho _{1v}:v\in \langle x_{2},x_{3}\rangle \rangle $ has a
fixed point. Note that $\mathrm{Aut}(F_{2})$ normalizes $F_{2}.$ Theorem \ref%
{1.3} shows that $\mathrm{Aut}(F_{2})$ has a fixed. The semi-direct product $%
\mathrm{Aut}(F_{2})\ltimes F_{2}$ has a fixed point by Lemma \ref{le1}.
\end{proof}

\bigskip

\begin{proof}[Proof of Corollary \protect\ref{1.4}]
Let 
\begin{equation*}
S=\{\varepsilon _{2}\rho _{12},\varepsilon _{1}(23),\varepsilon
_{1}\varepsilon _{2}(12),(i,i+1),i=3,...,n-1,\varepsilon _{n}\},
\end{equation*}%
\begin{equation*}
S_{1}=\{\varepsilon _{2}\rho _{12},\varepsilon _{3},\varepsilon
_{1}(23),\varepsilon _{1}\varepsilon _{2}(12)\}.
\end{equation*}%
It is direct (see \cite{var}) to check that any $2$ elements in $S$ or $%
S_{1} $ generate a finite group and thus have a common fixed point. Note
that $\mathrm{Aut}(F_{n})$ contains $[\frac{n}{3}]$ copies of $\mathrm{Aut}%
(F_{3})$ along the diagonal. Since $S_{1}$ lies in $\mathrm{Aut}(F_{3})$, we
know that $\mathrm{Aut}(F_{n})$ contains $[\frac{n}{3}]$ copies of $S_{1},$
any two of which commute. Bridson's result (cf. Lemma \ref{3.4}) implies
that when $d<2[\frac{n}{3}],$ the set $S_{1}$ has a global fixed point.

We will prove that any $(d+1)$-element subset of $S$ has a common fixed
point. Inductively, assume that it is already proved for any $(k-1)$-element
subset $(k\geq 4)$. Let 
\begin{equation*}
S_{k}=\{\varepsilon _{2}\rho _{12},\varepsilon _{3},\varepsilon
_{1}(23),\varepsilon _{1}\varepsilon _{2}(12),(i,i+1),i=3,...,k-1\}.
\end{equation*}%
Note that any $k$-element subset of $S$ generates a finite group, except
possibly $S_{k}.$ Note also that $S_{k}$ is subset of $\mathrm{Aut}(F_{k})$
and $\mathrm{Aut}(F_{n})$ contains $[\frac{n}{k}]$ copies of $\mathrm{Aut}%
(F_{k})$ and thus of $S_{k}.$ Lemma \ref{3.4} implies that when $d<(k-1)[%
\frac{n}{k}]$ any finitely many elements of $S_{k}$ has a fixed point. Note
that 
\begin{equation*}
d<2[\frac{n}{3}]\leq (k-1)[\frac{n}{k}]
\end{equation*}%
for any $k\geq 4.$ Therefore, any $k$-element subset $S_{k}$ has a fixed
point. The Helly theorem (cf. Lemma \ref{helly}) implies that $\langle
S\rangle =\mathrm{Aut}(F_{n})$ has a fixed point.
\end{proof}

\section{Generalizations}

In this section, we generalize the theorems proved in previous sections to a
more general setting. For this purpose, we give the following definition.

\begin{definition}
Let $X$ be a topological space and $G$ be a subgroup of its homeomorphism
group $\mathrm{Homeo}(X).$ For an integer $n>0,$ the pair (transformation
group) $(X,G)$ is $n$-Helly good if

\begin{enumerate}
\item[(i)] any finite subgroup $H<G$ has a non-empty fixed point set $%
\mathrm{Fix}(H);$ and

\item[(ii)] any collection of finitely many finite subgroups $%
\{H_{i}\}_{i\in I}$ $(H_{i}<G)$ has a non-empty intersection $\cap \mathrm{%
Fix}(H_{i}),$ whenever $n+1$ of them has a nonempty intersection $\cap
_{j=1}^{n+1}\mathrm{Fix}(H_{i_{j}}).$
\end{enumerate}
\end{definition}

In such a case, we call $X$ $n$-Helly good with respect to the group $G.$

\begin{theorem}
Let $(X,G)$ be an $m$-Helly good pair. Suppose that $n>m$ and $R$ is a ring.
When the semi-direct product $R^{n}\rtimes E_{n}^{\prime }(R)$ acts on $X$
by homeomorphisms in $G,$ the abelian subgroup $\{(x_{1},x_{2},\cdots
,x_{n})\mid x_{i}\in \mathbb{Z}x\subset R\}$ has a fixed point for any $x\in
R$. In particular, any elementary matrix $e_{ij}(x)$ has a fixed point when $%
E_{n+1}^{\prime }(R)$ acts on $X$ by homeomorphisms in $G,$ for any $x\in R$
and $1\leq i\neq j\leq n.$
\end{theorem}

\begin{proof}
The proof is similar to that of Proposition \ref{prop}. For any $x\in R,$
let 
\begin{eqnarray*}
S_{x} &=&\{(x,\sigma _{1}),(12),(23),...,(n-1,n),\sigma _{n}\} \\
&\subseteq &R^{n}\rtimes E_{n}^{\prime }(R).
\end{eqnarray*}%
Let $H=(W,S)$ be the coxeter group assigned to the Coxeter matrix $(m_{ij})$
defined by $m_{ii}=2,m_{12}=4,m_{i,i+1}=3$ for $2\leq i\leq n$ and $%
m_{n,n+1}=4$ and all other entries $m_{ij}=\infty .$ Recall that 
\begin{equation*}
H=\langle s_{i}\in S\mid (s_{i}s_{j})^{m_{ij}}=1,1\leq i,j\leq n\rangle .
\end{equation*}%
Let $f:H\rightarrow \langle S_{x}\rangle $ be defined as%
\begin{eqnarray*}
s_{1} &\rightarrow &((x,0,\cdots ,0),\sigma _{1}), \\
s_{i+1} &\rightarrow &(i,i+1),1\leq i\leq n-1, \\
s_{n+1} &\rightarrow &\sigma _{n}.
\end{eqnarray*}%
It could be directly checked that $f$ is a group homomorphism. The Coxeter
graph of $(W,S)$ is a path consisting of $n+1$ vertices with the edges $%
(s_{1},s_{2}),(s_{n},s_{n+1})$ labled by $4.$ The subgraph spanned by any $n$
vertices is either a path of type $B_{n}$ or disjoint union of two paths of
type $B_{k}$. By the classification of finite Coxeter groups (cf. Lemma \ref%
{finite}), any spherical subgroup of rank $n$ is finite in $H$. This proves
that the subgroup generated by any $n$ elements is finite. The definition of 
$m$-Helly good spaces implies that $S$ has a common fixed point $x_{0}\in X$%
. Note that $\{(x_{1},x_{2},\cdots ,x_{n})\mid x_{i}\in \mathbb{Z}x\subset
R\}<\langle S\rangle .$

We view $R^{n}=\langle e_{12}(x_{1}),e_{13}(x_{2}),\cdots
,e_{1n}(x_{n-1})\mid x_{i}\in R\rangle <E_{n+1}^{\prime }(R).$ Therefore, $%
e_{12}(x)$ has a fixed point. Since any elementary matrix $e_{ij}(x)$ is
conjugate to $e_{12}(x),$ the last claim is proved.
\end{proof}

\bigskip

Bridson \cite{brid08} obtains `The Triangle Criterion' for group actions on $%
\mathbb{R}$-trees: If $\Gamma $ is generated by $A_{1}\cup A_{2}\cup A_{3}$
and $H_{ij}=\langle A_{i},A_{j}\rangle $ is finite for $i,j=1,2,3$, then $%
\Gamma $ has property $F\mathbb{R}$ (i.e., any isometric action of $\Gamma $
on an $\mathbb{R}$-tree has a fixed point). In particular, the automorphism
group of free group $\mathrm{Aut}(F_{n})$ and the special linear group $%
\mathrm{SL}_{n}(\mathbb{Z}),n\geq 3,$ satisfy the triangle conditions and
thus have property $F\mathbb{R}.$ We prove a similar result for $1$-Helly
good transformation group $(X,G)$:

\begin{theorem}
\label{hellyfix}Let $(X,G)$ be a $1$-Helly good pair. If $\Gamma $ is
generated by $A_{1}\cup A_{2}\cup A_{3}$ and $H_{ij}=\langle
A_{i},A_{j}\rangle $ is finite for $i,j=1,2,3$, then any $G$-action of $%
\Gamma $ on $X$ has a fixed point. In other words, any group homomorphism $%
f:\Gamma \rightarrow X$ has a fixed point. In particular, any $G$-action of $%
\mathrm{Aut}(F_{n})$ (or $\mathrm{SL}_{n}(\mathbb{Z}),n\geq 3$) on $X$ has a
fixed point.
\end{theorem}

\begin{proof}
Since each $H_{ij}$ is finite, it has a fixed point by condition (i) in the
definition of $n$-Helly good pair. The condition (ii) implies the existence
of a global fixed point.
\end{proof}

Now we discuss some examples, which are $n$-Helly good.

\begin{example}
A $\mathrm{CAT}(0)$ space $X$ of dimension $n$ is $n$-Helly good, with
respect to the isometry group $G=\mathrm{Isom}(X).$
\end{example}

\begin{proof}
This follows easily the Helly's theorem (see Lemma \ref{helly}) and the
Bruhat-Tits fixed point theorem (see \cite{bh}, Corollary 2.8, p.179).
\end{proof}

\bigskip

Recall that a dendrite is a connected compact metrizable space $X$
satisfying one of the following equivalent conditions:

(1) Any two distinct points of $X$ can be separated by a point.

(2) $X$ is locally connected and contains no simple closed curve.

(3) The intersection of any two connected subsets of X remains connected.

(4) $X$ is a one-dimensional absolute retract.

(5) $C(X)$ is projective in the category of unital $C^{\ast }$-algebras.

(6) Any two points are extremities of a unique arc (i.e. a homeomorphic
image of a compact interval in the real line) in $X.$

For more details, see \cite{cc}\cite{cd}\cite{na}. We need the following
result.

\begin{lemma}
\label{den}Let $X$ be a dendrite. Let $G$ be a finite group. Then any action
of $G$ on $X$ has a fixed point.
\end{lemma}

\begin{proof}
When $G$ is finite, this is well-known (for example, see \cite{dm},
Corollary 5.1).
\end{proof}

A dendrite $X$ is actually $1$-Helly good, with respect to the homeomorphism
group $\mathrm{Homeo}(X).$ This is a special case of the following result. A
uniquely arcwise connected space $X$ is a hausdorff topological space in
which every pair of distinct points $x,y$ are joined by a unique arc $[x,y]$%
, which is a subset homeomorphic to a closed real interval. For more
information, see Bowditch \cite{bow}.

\begin{lemma}
\label{unique}Let $X$ be a uniquely arcwise connected space. Then $X$ is $1$%
-Helly good, with respect to the homeomorphism group $\mathrm{Homeo}(X).$
\end{lemma}

\begin{proof}
Let $H<\mathrm{Homeo}(X)$ be a fintie subgroup. Fix $x\in X.$ Since the
orbit $Hx$ is finite and there is a unique arc $[x,hx]$ connecting $x$ and $%
hx$ for each $h\in H,$ the union $\cup _{h\in H}[x,hx]$ is a compact tree
and thus a dendrite (cf. \cite{bow}, Lemma 1.2). Note that the union is
invariant under the action of $H.$ Lemma \ref{den} says that any action of $%
H $ has a fixed point. Therefore, the action of $H$ on $X$ has a non-empty
fixed point set $\mathrm{Fix}(H)$. This proves the condition (i).

For the condtion (ii), let $H_{1},H_{2},\cdots ,H_{n}$ be a collection of
finitely many finite subgroups. For any two points $x,y\in \mathrm{Fix}%
(H_{i}),$ the unique arc $[x,y]$ is invariant under the action of $H_{i}.$
Since a finite group cannot act freely on the real line $\mathbb{R},$ every
point in $[x,y]$ is a fixed point of $H_{i}.$ This proves that the fixed
point set $\mathrm{Fix}(H_{i})$ is path connected. For each $i,j=1,2,\cdots
,n,$ choose $x_{i,j}\in \mathrm{Fix}(H_{i})\cap \mathrm{Fix}(H_{j}).$ Denote
by $A=\{x_{ij}:i,j=1,2,\cdots ,n\}$ and $Y=\cup _{x,y\in A}[x,y],$ which is
a compact tree. The result follows from Helly's theorem for trees where
connectedness coincides with convexity.
\end{proof}

\bigskip

Theorem \ref{last} is a corollary of Lemma \ref{unique} and Theorem \ref%
{hellyfix}.

\bigskip

Department of Mathematical Sciences, Xi'an Jiaotong-Liverpool University,
111 Ren Ai Road, Suzhou, Jiangsu 215123, China.

E-mail: Shengkui.Ye@xjtlu.edu.cn

\end{document}